\documentclass[12pt,reqno]{amsart}

\usepackage{xcolor,tikz,pgfplots,graphicx,float,yfonts,multicol,stmaryrd,xfrac,relsize,bm,centernot}
\usepackage[shortlabels]{enumitem}
\usepackage[hidelinks]{hyperref}
\usepackage[capitalise]{cleveref}
\usetikzlibrary{cd}
\usetikzlibrary{trees}
\usepackage{amssymb}
\usepackage{amsthm}
\usepackage{amsmath}
\usepackage{a4wide}
\usepackage{latexsym}
\usepackage{mathrsfs}
\usepackage[T1]{fontenc}
\usepackage[latin1]{inputenc}

\newtheorem{theorem}{Theorem}[section]
\newtheorem{lemma}[theorem]{Lemma}
\newtheorem{proposition}[theorem]{Proposition}
\newtheorem{corollary}[theorem]{Corollary}

\theoremstyle{definition}

\newtheorem{remark}[theorem]{Remark}

\numberwithin{equation}{section}

\newcommand{\N}{\mathbb{N}}

\newcommand{\mcb}{\mathscr{B}}
\newcommand{\mcg}{\mathscr{G}}
\newcommand{\mck}{\mathscr{K}}
\newcommand{\mci}{\mathscr{I}}
\newcommand{\mcj}{\mathscr{J}}

\newcommand{\mcl}{\mathscr{L}}
\newcommand{\supp}{\operatorname{supp}}

\tikzcdset{scale cd/.style={every label/.append style={scale=#1},
    cells={nodes={scale=#1}}}}
\makeatletter
\newcommand\notni{\mathrel{\m@th\mathpalette\canc@l\owns}}
\newcommand\canc@l[2]{{\ooalign{$\hfil#1/\mkern1mu\hfil$\crcr$#1#2$}}}
\makeatother

\newcommand{\smashw}[2][l]{{\text{\makebox[0pt][#1]{$#2$}}}}

\renewcommand{\leq}{\ensuremath{\leqslant}}
\renewcommand{\le}{\ensuremath{\leqslant}}

\renewcommand{\ge}{\ensuremath{\geqslant}}

\subjclass[2010]{46H10,
47L10,
(primary); 
46B25,
46B26,
46B45,   	
47L20}
\keywords{Banach space, long sequence space, bounded
  operator, closed operator ideal, ideal lattice}

\title[Classifying closed ideals of operators on classical Banach spaces]{Classifying the closed ideals of bounded operators on two families of
  non-separable  classical Banach spaces}
\author[M.~Arnott]{Max Arnott}
\address{Department of Mathematics and Statistics, Fylde College, Lancaster University, Lancaster,
LA1 4YF, United Kingdom\medskip\mbox{}}
\email[Arnott]{m.arnott@lancaster.ac.uk}

\author[N.~J.~Laustsen]{Niels Jakob Laustsen}
\email[Laustsen]{n.laustsen@lancaster.ac.uk}
\begin{document}

\begin{abstract}
  \noindent We classify the closed ideals of bounded operators acting on the Banach spaces $\left(\bigoplus_{n \in \mathbb{N}} \ell_2^n\right)_{c_0} \oplus c_0(\Gamma)$ and $\left(\bigoplus_{n \in \mathbb{N}} \ell_2^n\right)_{\ell_1} \oplus \ell_1(\Gamma)$ for every uncountable cardinal~$\Gamma$.
\end{abstract}

\maketitle

\section{Introduction}
\noindent
Very few  Banach spaces $X$ are known for which the lattice of closed ideals of the Banach algebra $\mcb(X)$ of all bounded operators on~$X$ is fully understood. When~$X$ is finite-dimen\-sional, $\mcb(X)$ is simple, meaning that it contains no non-zero, proper ideals,
so we shall henceforth discuss infinite-dimensional Banach spaces only.

Our focus is on the ``classical'' case, that is,
Banach spaces that can be defined by elementary  means and/or were known to Banach and his contemporaries. There are currently only two  families of Banach spaces of this kind whose lattices of closed operator ideals are fully understood:
\begin{enumerate}[label={\normalfont{(\roman*)}}]
\item\label{daws}
Daws \cite[Theorem~7.4]{MD} has shown that for $X= c_0(\Gamma)$ or $X= \ell_p(\Gamma)$, where~$\Gamma$ is an  infinite cardinal and $1\le p<\infty$, the lattice of closed ideals of~$\mcb(X)$  is
\begin{align}\label{Eq:ClosedIdealsofc0Gamma}
    \{0\} \subsetneq \mck(X) \subsetneq \mck_{\aleph_1}(X) \subsetneq &\cdots  \subsetneq \mck_\kappa(X) \subsetneq \mck_{\kappa^+}(X)\subsetneq\cdots\notag\\ &\cdots\subsetneq \mck_\Gamma(X) \subsetneq \mck_{\Gamma^+}(X)=\mcb(X)\;.
\end{align} Here,
$\mck_\kappa(X)$ denotes the ideal of $\kappa$-compact operators for an uncountable cardinal~$\kappa$ (see page~\pageref{Defn:kappacompact} for the precise definition),   and 
$\kappa^+$ denotes the cardinal successor of $\kappa$. An alternative description of the closed ideals of~$\mcb(X)$  is given in \cite[Theorem~1.5]{JKS}; see also \cite[Theorem~3.7]{HK}.

\noindent Daws' theorem generalizes and unifies previous results of Calkin~\cite{Calkin} for $X=\ell_2$, Goh\-berg, Markus and Feldman for $X=c_0$ or $X=\ell_p$, $1\le p<\infty$, and Gramsch~\cite{Gr} and Luft~\cite{luft} independently for $X=\ell_2(\Gamma)$, where~$\Gamma$ is an arbitrary infinite cardinal.
\item Let $E= \left( \bigoplus_{n \in \mathbb{N}} \ell_2^n\right)_D$ for $D=c_0$ or $D=\ell_1$. Then it was shown in~\cite{LLR} and~\cite{LSZ}, respectively,
  that the lattice of closed ideals of~$\mcb(E)$  is
\begin{equation}\label{Eq:ClosedIdealsofBE}
  \{0\} \subsetneq \mck(E) \subsetneq \overline{\mcg_D}(E) \subsetneq \mcb(E) \;,\end{equation}  
where  $\overline{\mcg_D}(E)$ denotes the closure of the ideal of operators on~$E$ that factor through the space~$D$.
\end{enumerate}

We shall combine the above results to obtain two new ``hybrid'' families of Banach spaces, name\-ly~$\left(\bigoplus_{n \in \mathbb{N}} \ell_2^n\right)_{c_0} \oplus c_0(\Gamma)$ and its dual space $\left(\bigoplus_{n \in \mathbb{N}} \ell_2^n\right)_{\ell_1} \oplus \ell_1(\Gamma)$, for any uncountable cardinal~$\Gamma$, whose closed ideals of operators we classify. The precise statement is as follows. 

\begin{theorem}\label{mainthm}
  Let $(D,D_\Gamma) = (c_0,c_0(\Gamma))$ or $(D,D_\Gamma) = (\ell_1,\ell_1(\Gamma))$ for an uncountable cardinal~$\Gamma$, and set $E = \left(\bigoplus_{n \in \mathbb{N}} \ell_2^n\right)_D$ and $X= E\oplus D_\Gamma$. Then  the lattice of closed ideals of~$\mcb(X)$ is 
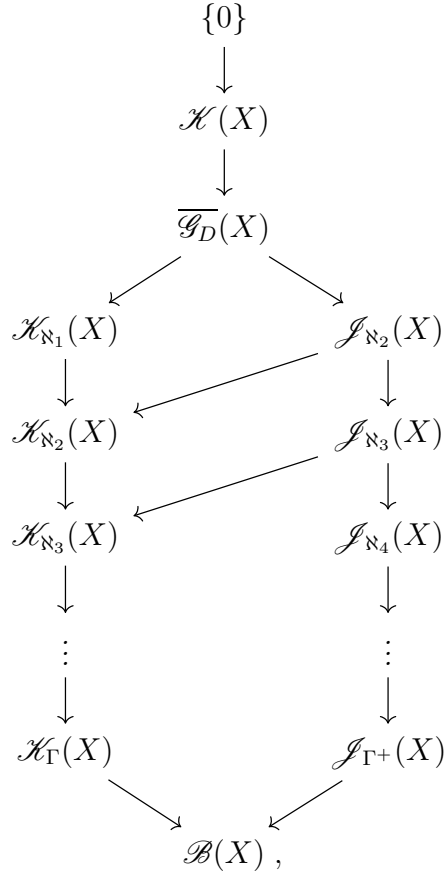
\begin{center}
    \begin{tikzcd}[tips=true,column sep=1em,row sep=1.5em]
\;&\{0\}\ar{d}&\;\\
\;&\mck(X)\ar{d} &\;\\
\;&\overline{\mcg_{D}}(X) \ar{dl} \ar{dr}&\;\\
\mck_{\aleph_1}(X)\ar{d} &\;& \mcj_{\aleph_2}(X)\ar{dll}\ar{d}\\
\mck_{\aleph_2}(X)\ar{d} & \;&\mcj_{\aleph_3}(X)\ar{d}\ar{dll}\\
\mck_{\aleph_3}(X)\ar{d} & \;& \mcj_{\aleph_4}(X)\ar{d}\\
\vdots\ar{d}&\;&\vdots\ar{d}\\
\mck_\Gamma(X)\ar{dr} & \; & \mcj_{\Gamma^+}(X)\ar{dl}\\
\;&\mcb(X)\smashw{\;,}&\;
\end{tikzcd}
\end{center}
where
\begin{align} \mcj_\kappa(X)= \biggl\{
\begin{pmatrix} T_{1,1} & T_{1,2}\\
T_{2,1} & T_{2,2}
\end{pmatrix}\in\mcb(X):
T_{1,1}&\in\overline{\mcg_{D}}(E),\, T_{1,2}\in\mcb(D_\Gamma;E),\notag\\[-2mm] T_{2,1}&\in\mcb(E;D_\Gamma),\, T_{2,2}\in\mck_\kappa(D_\Gamma)\biggr\}\label{Eq:mainthm} \end{align}
for each cardinal $\aleph_2\le\kappa\le\Gamma^+,$
where 
an arrow from an ideal~$\mci$ pointing to an ideal~$\mcj$ denotes that $\mci\subsetneq \mcj$ and there are no closed ideals of $\mcb(X)$ strictly contained in between~$\mci$ and~$\mcj$.
\end{theorem}

\begin{remark}\label{remark:idealclass} In addition to the ``classical'' Banach spaces listed above, there are a number of ``exotic'', or purpose-built, Banach spaces whose closed ideals of operators can be classified. They occur in two classes:
  \begin{itemize}
  \item The first class contains the famous
        Banach space of Argyros and Haydon~\cite{ah} solving the scalar-plus-compact problem, as well as several variants and descendants of it obtained in \cite{tarb,mpz,KLindiana}. 
      \item The other class consists of the Banach space~$C(K)$ of continuous, scalar-valued functions defined on Koszmider's Mr\'{o}wka space~$K$, as shown in~\cite[Theorem~5.5]{KK}. Koszmider's original    construction
                of~$K$ in~\cite{kosz}
assumed the Continuum Hypothesis. A~con\-struc\-tion within ZFC is given in~\cite{PKNJL}.
  \end{itemize}  

  A possible explanation for the scarcity of Banach spaces~$X$   whose closed ideals of opera\-tors have
  been classified, especially among classical spaces, is that recent research has shown that in many
  cases~$\mcb(X)$ has~$2^{\mathfrak{c}}$ closed ideals,
  where~$\mathfrak{c}$ denotes the cardinality of the contiuum. Note that this is the largest possible number of closed ideals for  separable~$X$. Spaces for which~$\mcb(X)$ has~$2^{\mathfrak{c}}$ closed ideals  include $X=L_p[0,1]$ for $p\in(1,\infty)\setminus\{2\}$ (see~\cite{JS}),  $X= \ell_p\oplus\ell_q$  for $1\le p<q\le\infty$ with $(p,q)\ne(1,\infty)$ and $X = \ell_p\oplus c_0$ for $1< p<\infty$ (see \cite{FSZ1,FSZ2}).

For several other spaces~$X$, it is known that~$\mcb(X)$  contains at least continuum many closed ideals. This includes  $X=L_1[0,1]$, $X=C[0,1]$  and~$X=L_\infty[0,1]$ (see~\cite{JPS}; note that these results also cover
$X=\ell_\infty$ because~$\ell_\infty$ and~$L_\infty[0,1]$ are isomorphic
as Banach spaces by~\cite{Pel}), as well as the Tsirelson space and the Schreier space of order $n\in\N$ (see~\cite{BKL}). For $X=\ell_1\oplus c_0$, the best known result is that~$\mcb(X)$  has at least uncountably many closed ideals (see~\cite{SW}).
\end{remark}

\section{Preliminaries}

\noindent Below we explain this paper's most important notation, which is mostly standard. All of our results are valid for both real and complex Banach spaces; we write $\mathbb{K}$ for the scalar field. By an \textit{operator} we always mean a bounded and linear map between normed spaces. 

\subsection*{Operator ideals} Following Pietsch~\cite{pie}, 
an \textit{operator ideal} is an assignment $\mci$ which de\-signates to each pair~$(X,Y)$ of Banach spaces a subspace~$\mci(X;Y)$ of the Banach space~$\mcb(X;Y)$ of operators from~$X$ to~$Y$ for which:
\begin{enumerate}[label={\normalfont{(\roman*)}}]
    \item there exists a pair $(X,Y)$ of Banach spaces for which $\mci(X;Y) \neq \{0\}$;
    \item for any quadruple $(W,X,Y,Z)$ of Banach spaces and any operators $S \in \mcb(W;X)$, $T \in \mci(X;Y)$, $U \in \mcb(Y;Z)$, we have that $UTS \in \mci(W;Z)$.
\end{enumerate}
As usual, we write $\mci(X)$ to abbreviate $\mci(X;X)$. For any operator ideal $\mci$, the map~$\overline{\mci}$ sending a pair of Banach spaces $(X,Y)$ to the norm closure of $\mci(X;Y)$ in $\mcb(X;Y)$ is also an operator ideal. If $\mci=\overline{\mci}$, then we call $\mci$ a \textit{closed operator ideal}. 

We shall consider the following three operator ideals:
\begin{itemize}
\item[$\mck$,] the ideal of \textit{compact operators.}
\item[$\mck_\kappa$,] the ideal of $\kappa$\textit{-compact operators,}\label{Defn:kappacompact}
  defined for any infinite cardinal~$\kappa$.

  \noindent The precise definition is as follows. An operator $T\in\mcb(X;Y)$ is $\kappa$\textit{-compact} if, for each $\epsilon > 0$, the
  closed unit ball~$B_X$ of~$X$ contains a subset $X_\epsilon$ with $|X_\epsilon|<\kappa$ such that \[\inf\{\|T(x-y)\|:y \in X_\epsilon \} \leq \epsilon\] for every~$x\in B_X$.
  Writing $\mck_\kappa(X;Y)$ for the set of $\kappa$-compact operators from~$X$ to~$Y$, we obtain a closed operator ideal~$\mck_\kappa$.
  
\noindent Notice that $\mck_{\aleph_0}(X;Y)=\mck(X;Y)$, so the notion of $\kappa$-compactness is indeed a generalisation of compactness.
\item[$\mcg_D$,] the ideal of operators factoring through a certain Banach space~$D$. 
  
\noindent  Here, we say that an operator $T \in \mcb(X;Y)$ \textit{factors through} $D$ if there are operators $U\colon X\to D$ and $V\colon D\to Y$ such that $T=VU$, and we write $\mcg_D(X;Y)$ for the set of operators factoring through~$D$.  This defines 
an operator ideal $\mcg_D$ provided that~$D$ contains a complemented subspace isomorphic to~$D\oplus D$, which is true in the cases that we shall consider, namely $D=c_0$ 
 and  $D=\ell_p$ for some $p\in[1,\infty)$.
\end{itemize}

\subsection*{Operators on the direct sum of a pair of Banach spaces}
Let $X_1$ and $X_2$ be Banach spaces, and endow their direct sum~$X_1 \oplus X_2$ with any norm satisfying
\[ \max\{\lVert x_1\rVert,\lVert x_2\rVert\}\le\lVert (x_1,x_2)\rVert\le \lVert x_1\rVert+\lVert x_2\rVert\qquad (x_1\in X_1,\, x_2\in X_2)\;. \]
For $i \in \{1,2\}$, we write $Q_i\colon X_1 \oplus X_2 \to X_i$ and $J_i\colon X_i \to X_1 \oplus X_2$ for the $i^{\text{th}}$ coordinate projection and embedding, respectively. For  $T \in\mcb(X_1 \oplus X_2)$ and $i,j \in \{1,2\}$, set $T_{i,j}= Q_iTJ_j\in\mcb(X_j;X_i)$. Then we have 
\begin{equation}\label{Eq:MatrixT}
  T=\sum_{i,j=1}^2 J_i T_{i,j} Q_j\;. \end{equation}
It follows that, for any operator ideal~$\mci$,
\begin{equation}\label{Eq:Prop2.2} T\in\mci(X_1\oplus X_2)\quad \iff\quad T_{i,j}\in\mci(X_j;X_i)\ \text{for}\ i,j\in\{1,2\}\;. \end{equation}
 We shall identify the operator $T\in\mcb(X_1 \oplus X_2)$ with the matrix
 \[  T= \begin{pmatrix} T_{1,1} & T_{1,2}\\
T_{2,1} & T_{2,2}
\end{pmatrix}\;. \]
Then the action of $T$ on a pair $(x_1,x_2) \in X_1 \oplus X_2$ is given by 
\[T(x_1,x_2)= \begin{pmatrix} T_{1,1} & T_{1,2}\\
T_{2,1} & T_{2,2}
\end{pmatrix}\begin{pmatrix} x_1\\ x_2\end{pmatrix} = \begin{pmatrix} T_{1,1}x_1+ T_{1,2}x_2\\
    T_{2,1}x_1 + T_{2,2}x_2 \end{pmatrix}\;.\] Composition of operators is given by matrix multiplication, addition is entry\-wise, and the operator norm  satisfies
  \begin{equation}\label{Eq:OpNormSum}
    \max_{i,j\in\{1,2\}}\lVert T_{i,j}\rVert\le \lVert T\rVert\le \sum_{i,j=1}^2\lVert T_{i,j}\rVert\;.
  \end{equation}  

For a subset $\mci$ of $\mcb(X_1 \oplus X_2)$ and $i,j \in \{1,2\}$,  we define the $(i,j)^{\text{th}}$ \textit{quadrant} of~$\mci$ by
\begin{equation}\label{Eq:quadrants} \mci_{i,j} = \{ Q_iTJ_j : T\in\mci\}\subseteq\mcb(X_j;X_i)\;.
\end{equation}
On the other hand, given subsets $\mci_{i,j}$ of $\mcb(X_j;X_i)$ for  $i,j \in \{1,2\}$, we define
\begin{equation}\label{Eq:matrixofsets} \begin{pmatrix} \mci_{1,1} & \mci_{1,2}\\ \mci_{2,1} & \mci_{2,2}  \end{pmatrix} = \left\{ \begin{pmatrix} T_{1,1} & T_{1,2}\\
T_{2,1} & T_{2,2}
\end{pmatrix}: T_{i,j} \in \mci_{i,j}\ (i,j \in \{1,2\})\right\}\subseteq\mcb(X_1 \oplus X_2)\;.   \end{equation}
The first part of the following lemma can be seen as a generalisation of~\eqref{Eq:Prop2.2}  to the case where the ideal~$\mci$ does not come from an operator ideal. It
says that if we decompose~$\mci$ into its quadrants according to~\eqref{Eq:quadrants} and then reassemble the quadrants according to~\eqref{Eq:matrixofsets}, we obtain~$\mci$.

\begin{lemma}\label{LemmaDiagonalIdeal} Let $\mci$ be an ideal of $\mcb(X_1\oplus X_2)$ for some  Banach spaces~$X_1$ and~$X_2$.
  Then
\begin{equation}\label{LemmaDiagonalIdealEq1}  \mci = \begin{pmatrix} \mci_{1,1} & \mci_{1,2}\\ \mci_{2,1} & \mci_{2,2} \end{pmatrix}, \end{equation}
and  $\mci_{i,i}$ is an ideal of  $\mcb(X_i)$ for $i\in\{1,2\}$. Moreover,   $\mci_{i,j}$ is closed in $\mcb(X_j,X_i)$ for each $i,j\in\{1,2\}$ if and only if~$\mci$ is closed in $\mcb(X_1\oplus X_2)$.  
\end{lemma}
\begin{proof} The inclusion $\subseteq$ in~\eqref{LemmaDiagonalIdealEq1} holds true by the  definitions.

  Conversely, suppose that $T=(T_{i,j})_{i,j=1}^2$ with $T_{i,j}\in\mci_{i,j}$ for each $i,j\in\{1,2\}$, say $T_{i,j}=Q_iS^{i,j}J_j$, where $S^{i,j}\in\mci$. Then, by~\eqref{Eq:MatrixT}, we have
  \[ T  =  \sum_{i,j=1}^2 J_iT_{i,j}Q_j = \sum_{i,j=1}^2 (J_iQ_i)S^{i,j}(J_jQ_j)\in\mci  \]
    because $\mci$ is an ideal of~$\mcb(X_1\oplus X_2)$ and $J_kQ_k\in\mcb(X_1\oplus X_2)$ for $k\in\{1,2\}$.   

 Next, we verify that $\mci_{i,i}$ is an ideal of  $\mcb(X_i)$ for $i\in\{1,2\}$. It is clear that $\mci_{i,i}$ is a subspace.
 Suppose that $S\in\mci_{i,i}$ and 
  $T\in\mcb(X_i)$, say  $S = U_{i,i}$, where $U\in\mci$. Then $UJ_iTQ_i\in\mci$ because $\mci$ is an ideal of~$\mcb(X_1\oplus X_2)$ and
 $J_iTQ_i\in\mcb(X_1\oplus X_2)$, and hence 
 \[ \mci_{i,i}\ni  (UJ_iTQ_i)_{i,i}  = Q_iUJ_iTQ_iJ_i =  ST\;. \] 
 The proof that $TS\in  \mci_{i,i}$ is similar.

The final clause follows easily from~\eqref{Eq:OpNormSum} and~\eqref{LemmaDiagonalIdealEq1}.
\end{proof}   

\subsection*{Long sequence spaces}
For a non-empty set~$\Gamma$ and $p\in[1,\infty)$, we consider the Banach spaces
  \begin{align*}
    c_0(\Gamma) &= \left\{x\in \mathbb{K}^\Gamma : \{\gamma \in \Gamma : \lvert x(\gamma)\rvert>\epsilon \}\ \text{is finite for every}\ \epsilon >0\right\}\\
    \intertext{and}
    \ell_p(\Gamma) &= \Bigl\{ x\in\mathbb{K}^\Gamma : \sum_{\gamma\in\Gamma}\lvert x(\gamma)\rvert^p<\infty\Bigr\}\;.
  \end{align*}
 The norm of $c_0(\Gamma)$ is the supremum norm, and the norm of $\ell_p(\Gamma)$ is $\|x\|=\bigl(\sum_{\gamma\in\Gamma}| x(\gamma)|^p\bigr)^{\frac{1}{p}}$.
As usual, we write~$c_0$ and~$\ell_p$ instead of~$c_0(\N)$ and~$\ell_p(\N)$, respectively.  
For notational convenience, we use the convention that~$D_\Gamma$ will denote either~$c_0(\Gamma)$ or~$\ell_p(\Gamma)$ for some $p\in[1,\infty)$, unless otherwise specified. Only the cardinality of the index set~$\Gamma$ matters, in the sense that $D_\Gamma$ is isometrically isomorphic to $D_{\lvert\Gamma\rvert}$, and $D_\Gamma$ is not isomorphic to~$D_\Delta$ for any index set~$\Delta$ of cardinality other than~$\lvert\Gamma\rvert$. 
  
  The \textit{support} of an element $x\in D_\Gamma$ is $\supp x = \{\gamma \in \Gamma : x(\gamma)\neq 0\}$,  which is always a countable set. 

  For $\gamma\in\Gamma$, $e_\gamma\in D_\Gamma$ denotes the element given by $e_\gamma(\beta)=1$ if $\beta=\gamma$ and  $e_\gamma(\beta)=0$ otherwise.  The (transfinite) sequence $(e_\gamma)_{\gamma\in\Gamma}$ is the \textit{(long) unit vector basis} for~$D_\Gamma$.
   The only facts about it that we shall use are that $(e_\gamma)_{\gamma\in\Gamma}$ spans a dense  subspace of~$D_\Gamma$, and that the existence of such a basis implies that~$D_\Gamma$ has the approximation property.
  
For a subset~$\Delta$ of~$\Gamma$, $P_{\Delta}\in\mcb(D_\Gamma)$ denotes the \textit{basis projection} given by
$(P_{\Delta} x)(\gamma) = x(\gamma)$ if $\gamma\in\Delta$ and $(P_{\Delta}x)(\gamma)= 0$ otherwise, for every $x\in D_\Gamma$. 

\section{The proof of \cref{mainthm}}
\noindent To aid the presentation,  we split
the proof of \cref{mainthm} into a series of lemmas, some of which may essentially be known. However, to keep the the presentation as self-contained as possible, we include full proofs, except for references to the ideal classifications~\cite{MD,LLR} that we will ultimately need anyway. The  proof of \cref{mainthm} itself
requires results about the transfinite sequence spaces~$c_0(\Gamma)$ and~$\ell_1(\Gamma)$ only, not $\ell_p(\Gamma)$ for $1<p<\infty$. However, our first few results hold true also for the latter spaces and with identical proofs, so we give these more general results.

\begin{lemma}\label{Lemma:FactorThroughD} Let $D_\Gamma=c_0(\Gamma)$ or $D_\Gamma=\ell_p(\Gamma)$ for some $p\in[1,\infty)$ and some set~$\Gamma\ne\emptyset$. 
\begin{enumerate}[label={\normalfont{(\roman*)}}]
\item\label{Lemma:FactorThroughD1} Every separable subspace of~$D_\Gamma$ is contained in the image of the basis projection~$P_\Delta$ for some  countable 
  subset~$\Delta$ of~$\Gamma$.
\item\label{Lemma:FactorThroughD2} Suppose that $D_\Gamma\ne\ell_1(\Gamma)$. Then, for every Banach space~$E$  and every operator $T\colon D_\Gamma\to E$  for which there exists an injective operator from the image of~$T$ into~$\ell_\infty,$ there is a countable 
  subset~$\Delta$ of~$\Gamma$ such that $T =  TP_\Delta$.
\end{enumerate}
\end{lemma}

\begin{proof} \ref{Lemma:FactorThroughD1}. 
Every separable subspace~$E$ of~$D_\Gamma$ has the form
$E=\overline{W}$ for some countable subset~$W$ of~$D_\Gamma$. Define
$\Delta = \bigcup_{w\in W} \operatorname{supp} w$, which 
 is a countable union of countable sets and is thus countable. The continuity
of the projection~$P_\Delta$ implies that $x=P_\Delta x$ for every
$x\in E$. Hence the image of~$P_{\Delta}$ contains~$E$.

\ref{Lemma:FactorThroughD2}. 
Let $U\colon T(D_\Gamma)\to \ell_\infty$ be  an injective operator. Assume towards a contradiction that the set \[ \Delta_{k,m} = \Bigl\{\gamma \in \Gamma: \lvert UTe_\gamma(m)\rvert\ge\frac{1}{k} \Bigr\} \]
is infinite for some $k,m\in\N$, so that it contains an infinite sequence  $(\gamma_n)_{n\in\N}$ of distinct elements. For each $n\in\N$,  take a scalar~$\sigma_n$ of modulus one 
such that \mbox{$\sigma_n\cdot (U T e_{\gamma_n})(m)\ge 1/k$}. Since $D_\Gamma\ne\ell_1(\Gamma)$, we have $x= \sum_{n \in \mathbb{N}}\frac{\sigma_n}{n}e_{\gamma_n}\in D_\Gamma$, but
\[(UTx)(m) = \sum_{n \in \mathbb{N}} \frac{\sigma_n}{n}(UTe_{\gamma_n})(m)\ge \frac{1}{k}\sum_{n\in\mathbb{N}} \frac{1}{n} = \infty\;,\] a contradiction. 
Hence $\Delta_{k,m}$ is finite for each $k,m\in\N$, so the union $\Delta = \bigcup_{k,m \in \mathbb{N}}\Delta_{k,m}$
is countable. For each $\gamma\in\Gamma\setminus\Delta$, we have $UTe_\gamma=0$, so $Te_\gamma=0$ by the injectivity of~$U$, and therefore $TP_\Delta = T$.
\end{proof}

\begin{remark}
\begin{enumerate}[label={\normalfont{(\roman*)}}]
\item  The case $D_\Gamma=\ell_1(\Gamma)$ must be excluded in \cref{Lemma:FactorThroughD}\ref{Lemma:FactorThroughD2} because, for every non-zero Banach space~$E$, there is an operator $T\colon\ell_1(\Gamma)\to E$ which has one-dimensional image and satisfies $Te_\gamma\ne 0$ for every $\gamma\in\Gamma$, namely
  the summation operator~$T$ given by $Tx = \sum_{\gamma\in\Gamma} x(\gamma)\, y$ for every $x\in\ell_1(\Gamma)$, where $y$ is any fixed non-zero element of~$E$.
\item We refer to~\cite{JK} for a detailed discussion of the condition in \Cref{Lemma:FactorThroughD}\ref{Lemma:FactorThroughD2} that  the image of~$T$ admits an injective operator into~$\ell_\infty$.
\end{enumerate}  
\end{remark}  

\begin{corollary}\label{Cor:FactorThroughD}
  Let $(D,D_\Gamma) = (c_0,c_0(\Gamma))$ or $(D,D_\Gamma) = (\ell_p,\ell_p(\Gamma))$ for some $p\in[1,\infty)$ and some uncountable set~$\Gamma$, and let~$E$ be any separable Banach space. 
    Then  \[ \mcb(D_\Gamma;E) = \mcg_D(D_\Gamma;E)\qquad\text{and}\qquad \mcb(E;D_\Gamma)=\mcg_{D}(E;D_\Gamma)\;. \]
\end{corollary}  

\begin{proof}
The first identity for $D_\Gamma\ne\ell_1(\Gamma)$, and the second identity in full generality, both follow easily from \cref{Lemma:FactorThroughD}  because 
the image of the projection~$P_\Delta$ for~$\Delta$ countable is either finite-dimensional or isomorphic to~$D$, and~$E$, being separable, embeds isometrically into~$\ell_\infty$. 

It remains to show that every operator $T\colon\ell_1(\Gamma)\to E$ factors through~$\ell_1$. We use the lifting property  of~$\ell_1$ (see for instance \cite[Theorem 5.1]{BST})
to verify this. Indeed, since~$E$ is separable, we can take a surjective operator $S\colon\ell_1\to E$. By the open mapping theorem,   there is a constant $c>0$ such that, for  every $y\in E$, there is $x\in\ell_1$ with $Sx=y$ and $\lVert x\rVert\le c\lVert y\rVert$. Hence, for each $\gamma\in\Gamma$, we can find $x_\gamma\in\ell_1$ such that $Sx_\gamma = Te_\gamma$ and $\lVert x_\gamma\rVert\le c\lVert Te_\gamma\rVert\le c\lVert T\rVert$. It follows that we can define an operator $R\colon\ell_1(\Gamma)\to\ell_1$ by $Re_\gamma = x_\gamma$ for each $\gamma\in\Gamma$, and clearly $T = SR$.
\end{proof}

\begin{lemma}\label{lemmaFactorIdofc0thruE}
 Let $D=c_0$ or $D=\ell_p$ for some $p\in[1,\infty)$, and let $(E_n)_{n\in\N}$ be a sequence of non-zero Banach spaces.   Then there are operators $R\colon D\to \left(\bigoplus_{n \in \mathbb{N}} E_n\right)_D$ and $S\colon \left(\bigoplus_{n \in \mathbb{N}} E_n\right)_D\to D$ such that $SR = I_D$.
\end{lemma}   
\begin{proof} For every $n\in\mathbb{N}$, 
  choose   $y_n\in E_n$ and $f_n\in E_n^*$ with \mbox{$\lVert y_n\rVert = \lVert f_n\rVert = \langle y_n,f_n\rangle =1$}, and define $R\colon (\lambda_n)\mapsto (\lambda_n y_n)$ for $(\lambda_n)\in D$ and $S\colon (x_n)\mapsto (\langle x_n,f_n\rangle)$ for $(x_n)\in\left(\bigoplus_{n \in \mathbb{N}} E_n\right)_D$.
\end{proof}

\begin{lemma}\label{Lemma:embedEinD}
  Let $D=c_0$ or $D=\ell_1$. Then~$D$ contains a subspace which is isomorphic to~$\bigl(\bigoplus_{n \in \mathbb{N}} \ell_2^n\bigr)_D$.
  \end{lemma}   
\begin{proof} 
  This follows by combining the
  fact that~$D$ contains almost isometric copies of~$\ell_2^n$ for  every $n\in\N$ with the fact that
$D$ is isomorphic to the $D$-direct sum of countably many copies of itself.
\end{proof}

\begin{lemma}\label{Lemma:4cases}
  Let $(D,D_\Gamma) = (c_0,c_0(\Gamma))$ or $(D,D_\Gamma) = (\ell_1,\ell_1(\Gamma))$ for some infinite set~$\Gamma,$ and set $E=\bigl(\bigoplus_{n \in \mathbb{N}} \ell_2^n\bigr)_D$. Then the identity operator on~$D$ factors through every
non-compact operator belonging to either~$\mcb(E)$, $\mcb(D_\Gamma)$, $\mcb(E;D_\Gamma)$ or~$\mcb(D_\Gamma;E)$.  
\end{lemma}

\begin{proof} Let $T$ be a non-compact operator. We examine each of the four cases separately:
\begin{enumerate}[label={\normalfont{(\roman*)}}]
\item\label{Lemma:4cases1} If $T\in\mcb(E)$, then~$I_D$ factors through~$T$ by \cite[Corollary~3.8 and Example~3.9]{LLR}.
\item\label{Lemma:4cases2} For $T\in\mcb(D_\Gamma)$,  a careful examination of the proofs of \cite[Proposition~4.3 and Theorems~6.2 and~7.3]{MD} shows that there are 
  operators $R,S\in\mcb(D_\Gamma)$
  such that $STR=P_\Delta$ for some infinite subset $\Delta$ of~$\Gamma$.
  Choose an infinite sequence $(\gamma_n)$ of distinct elements in~$\Delta$, and define operators 
$U\colon D\to D_\Gamma$ and $V\colon D_\Gamma\to D$ by $U(e_n)  = e_{\gamma_n}$ and $V(e_{\gamma_n}) =  e_n$ 
  for each $n\in\mathbb{N}$, and $V(e_\gamma)=0$ for $\gamma\in\Gamma\setminus\{\gamma_n:n\in\mathbb{N}\}$.
  Then we have $VSTRU=I_D$.
\item\label{Lemma:4cases3} For $T\in\mcb(E;D_\Gamma)$, \cref{Lemma:FactorThroughD}\ref{Lemma:FactorThroughD1} implies that  $T=P_{\Delta}T$ 
  for some countable sub\-set~$\Delta$ of~$\Gamma$.
  Note that~$\Delta$ is infinite, as other\-wise $P_{\Delta}T$ would be compact. Enumerate~$\Delta$ as $\{\gamma_n: n\in\N\}$. Then, defining the operators $U\colon D\to D_\Gamma$ and $V\colon D_\Gamma\to D$
  as in case~\ref{Lemma:4cases2}, we have $UV=P_\Delta$. 
  Choose operators $R\colon D\to E$ and $S\colon E\to D$ as in \Cref{lemmaFactorIdofc0thruE}, and observe that $RVT$  is non-compact, as other\-wise $(US)(RVT)= P_\Delta T=T$ would be compact. Now the conclusion follows by applying
  case~\ref{Lemma:4cases1} to the operator~$RVT\in\mcb(E)$.
\item\label{Lemma:4cases4} Finally, suppose that $T\in\mcb(D_\Gamma;E)$. By \cref{Lemma:embedEinD} and the fact that~$\Gamma$ is infinite, we can find an isomorphic embedding $U\in\mcb(E;D_\Gamma)$. Then~$UT$ is non-compact, and the conclusion follows by applying
  case~\ref{Lemma:4cases2} to the operator $UT\in\mcb(D_\Gamma)$. \qedhere
\end{enumerate}   
\end{proof}

\begin{remark}
\cref{Lemma:4cases} is also true for $(D,D_\Gamma) = (\ell_p,\ell_p(\Gamma))$ when $1<p<\infty$. However, $E=\bigl(\bigoplus\ell_2^n\bigr)_{\ell_p}$ is isomorphic to~$\ell_p$ in these cases, so only case~\ref{Lemma:4cases2} above would be non-trivial.
\end{remark}  

\begin{corollary}\label{Prop:idealdichotomy}
  Let $X = \left(\bigoplus_{n \in \mathbb{N}} \ell_2^n\right)_D\oplus D_\Gamma,$ where $(D,D_\Gamma) = (c_0,c_0(\Gamma))$ or $(D,D_\Gamma) = (\ell_1,\ell_1(\Gamma))$ for some infinite set~$\Gamma,$ and let $\mci$ be an ideal of $\mcb(X)$. Then either \mbox{$\mci\subseteq\mck(X)$} or $\mcg_{D}(X)\subseteq\mci$.
\end{corollary}

\begin{proof}
 For notational convenience, write $X = X_1\oplus X_2$, where $X_1=\left(\bigoplus_{n \in \mathbb{N}} \ell_2^n\right)_D$ and $X_2 = D_\Gamma$. 
   Suppose that  $\mci\nsubseteq\mck(X)$, and choose $T\in\mci\setminus\mck(X)$. Then~\eqref{Eq:Prop2.2} shows that  $T_{i,j}\notin\mck(X_j;X_i)$ for some $i,j\in\{1,2\}$. \cref{Lemma:4cases} implies that there are operators $U\colon D\to X_j$ and $V\colon X_i\to D$ such that 
   $VT_{i,j}U=I_{D}$.
Hence, for each $S = R_2R_1\in\mcg_D(X)$, where $R_1\in\mcb(X;D)$ and $R_2\in\mcb(D;X)$, we have
  \[ S = R_2VT_{i,j}UR_1= (R_2VQ_i)T(J_jUR_1)\in\mci \]
because~$\mci$ is an ideal of~$\mcb(X)$. This shows that $\mcg_D(X)\subseteq\mci$, as desired. 
\end{proof}

For a Banach space $X$, define 
\[ \Xi(X) = \{\mci: \mci\ \text{is a closed ideal of}\ \mcb(X)\  \text{and}\ \mci \supsetneq \mck(X)\}\;,         \] and order $\Xi(X)$ by inclusion.
For a pair of Banach spaces $X_1$ and $X_2$, we  endow the set $\Xi(X_1)\times\Xi(X_2)$ with the product order; that is, \[ (\mci_1,\mci_2) \leq (\mcj_1,\mcj_2)\ \iff\ [\mci_1\subseteq \mcj_1] \land [\mci_2 \subseteq \mcj_2]\;. \] 

\begin{proposition}\label{orderisomorphism}
  Let $X = E\oplus D_\Gamma$, where $E=\left(\bigoplus_{n \in \mathbb{N}} \ell_2^n\right)_D$ and 
  either $(D,D_\Gamma) = (c_0,c_0(\Gamma))$ or $(D,D_\Gamma) = (\ell_1,\ell_1(\Gamma))$ for some infinite set~$\Gamma$.
  The map \[ \xi\;\colon\ \Xi(E) \times \Xi(D_\Gamma) \to \Xi(X)\;;\ (\mci,\mcj)\mapsto     \begin{pmatrix}
  \mci  & \mcb(D_\Gamma;E) \\
    \mcb(E; D_\Gamma) & \mcj \end{pmatrix}\;, \]
is an order isomorphism.
\end{proposition}

\begin{proof} Recall from \cref{Cor:FactorThroughD}
  that $\mcb(E;D_\Gamma)=\mcg_{D}(E;D_\Gamma)$ and $\mcb(D_\Gamma;E) = \mcg_{D}(D_\Gamma;E)$, and that $\mcg_{D}(E)\subseteq\mci$
  and $\mcg_{D}(D_\Gamma)\subseteq\mcj$  for every 
  $(\mci,\mcj)\in \Xi(E)\times\Xi(D_\Gamma)$ by the ideal classifications~\eqref{Eq:ClosedIdealsofBE} and~\eqref{Eq:ClosedIdealsofc0Gamma}, respectively. Using these facts, one can easily verify that~$\xi(\mci,\mcj)$ is an ideal of~$\mcb(X)$ with $\mck(X)\subsetneq\xi(\mci,\mcj)$.  Moreover, $\xi(\mci,\mcj)$ is closed by \Cref{LemmaDiagonalIdeal}, so it belongs to~$\Xi(X)$.

  To see that~$\xi$ is surjective, let $\mcl\in\Xi(X)$. \Cref{LemmaDiagonalIdeal} shows that
  \[ \mcl = \begin{pmatrix} \mcl_{1,1} & \mcl_{1,2}\\ \mcl_{2,1} & \mcl_{2,2} \end{pmatrix}, \]
  where  $\mcl_{1,1}$ and $\mcl_{2,2}$ are closed ideals of~$\mcb(E)$ and~$\mcb(D_\Gamma)$, respectively. Moreover,  \Cref{Prop:idealdichotomy}  implies that $\mcg_{D}(X)\subseteq\mcl$, so by~\eqref{Eq:Prop2.2}, we have:
  \begin{itemize}
  \item $\mcl_{1,1}\supseteq\mcg_{D}(E)$, so $\mcl_{1,1}\in\Xi(E)$;
  \item $\mcl_{1,2}\supseteq\mcg_{D}(D_\Gamma;E)=\mcb(D_\Gamma;E)$, so $\mcl_{1,2}=\mcb(D_\Gamma;E)$, and similarly $\mcl_{2,1}=\mcb(E;D_\Gamma)$;
  \item $\mcl_{2,2}\supseteq\mcg_{D}(D_\Gamma)$,  so $\mcl_{2,2}\in\Xi(D_\Gamma)$.
    \end{itemize}
This verifies that $\mcl = \xi(\mcl_{1,1},\mcl_{2,2})$. 

Finally, working straight from the definitions, we see that
$(\mci_1,\mcj_1)\le(\mci_2,\mcj_2)$ if and only if $\xi(\mci_1,\mcj_1)\subseteq\xi(\mci_2,\mcj_2)$ for
$(\mci_1,\mcj_1),(\mci_2,\mcj_2)\in\Xi(E) \times \Xi(D_\Gamma)$. This shows first that~$\xi$ is injective and thus a bijection, and secondly that both~$\xi$ and its inverse   are order-preserving. 
\end{proof}

We can now prove  \cref{mainthm} easily. 

\begin{proof}[Proof of \cref{mainthm}]
  Both~$E$ and~$D_\Gamma$ have the approximation property, so also must their direct sum~$X$. Therefore~$\mck(X)$ is the smallest non-zero closed ideal of $\mcb(X)$. 
  \cref{orderisomorphism} shows that any other non-zero closed ideal~$\mcl$ of~$\mcb(X)$ has the form $\mcl = \xi(\mci,\mcj)$ for unique closed ideals $\mci\in \Xi(E)$ and $\mcj\in\Xi(D_\Gamma)$. By the ideal classifications~\eqref{Eq:ClosedIdealsofBE} and~\eqref{Eq:ClosedIdealsofc0Gamma},
    either
  $\mci= \overline{\mcg_{D}}(E)$ or $\mci=\mcb(E)$, while $\mcj = \mck_\kappa(D_\Gamma)$ for a unique cardinal $\aleph_1\le\kappa\le\Gamma^+$.

  Suppose first that $\mci= \overline{\mcg_{D}}(E)$. If $\kappa=\aleph_1$, then $\mcj = \overline{\mcg_{D}}(D_\Gamma)$, so $\mcl=\overline{\mcg_{D}}(X)$. Otherwise $\kappa\ge\aleph_2$ and $\mcl=\mcj_\kappa(X)$ in the notation of~\eqref{Eq:mainthm}.

  Next, suppose that $\mci=\mcb(E)$, which is equal to~$\mck_\kappa(E)$ because $E$ has density $\aleph_0<\kappa$. Hence we have
  $\mcl = \mck_\kappa(X)$. (Note that this is equal to~$\mcb(X)$ if $\kappa=\Gamma^+$.)
\end{proof}

\subsection*{Ackonowledgements} We are grateful to Dr Tomasz Kania (Czech Academy of Sciences and  Jagiellonian University) for his feedback on an earlier version of this manuscript.

\end{document}